\documentclass[12pt]{article}
\usepackage{listings}
\usepackage{amsmath,amssymb}
\usepackage{subcaption}
\usepackage{graphicx}
\usepackage{tikz}
\usepackage{structuralanalysis}
\usepackage{siunitx}
\usepackage{enumerate}
\usepackage{mathtools}
\usepackage{epic}
\usepackage{float}
\usepackage{mathtools}
\usepackage{authblk}
\usepackage{blindtext}
\usepackage[numbers]{natbib}
\usepackage{geometry}
\usepackage[hang,flushmargin]{footmisc}

\newcommand{\qed}{\hfill \mbox{\raggedright \rule{.07in}{.1in}}}

\newenvironment{proof}{\vspace{1ex}\noindent{\bf Proof}\hspace{0.5em}}
	{\hfill\qed\vspace{1ex}}
\newtheorem{theorem}{Theorem}
\newtheorem{example}{Example}

\newtheorem{observation}{Observation}
\newtheorem{definition}{Definition}

\newtheorem{corollary}{Corollary}

\usepackage{graphicx}

\providecommand{\keywords}[1]{\textbf{\textit{keywords:}} #1}

\date{}
\begin{document}

	\title{The $\gamma$-Signless Laplacian Adjacency Matrix of Mixed Graphs}
	

	\author{ Omar Alomari \thanks{ German Jordanian University, omar.alomari@gju.edu.jo}, Mohammad Abudayah \thanks{German Jordanian University, mohammad.abudayah@gju.edu.jo, corresponding author }, Manal Ghanem \thanks{The university of Jordan, m.ghanem@ju.edu.jo} }
		


\maketitle

\begin{abstract}
The $\alpha$-Hermitian adjacency matrix $H_\alpha$ of a mixed graph $X$ has been recently introduced. It is a generalization of the adjacency matrix of unoriented graphs. In this paper, we consider a special case of the complex number $\alpha$. This enables us to define an incidence matrix of mixed graphs. Consequently, we define a generalization of line graphs as well as a generalization of the signless Laplacian adjacency matrix of graphs. We then study the spectral properties of the signless Laplacian adjacency matrix of a mixed graph. Lastly, we characterize when the signless Laplacian adjacency matrix of a mixed graph is singular and give lower and upper bounds of number of arcs and digons in terms of largest and lowest eigenvalue of the signless Laplacian adjacency matrix.
\end{abstract}


\section{ \normalsize Introduction}
A mixed graph is a triple $(V(X),E(X),\vec{E}(X))$, where $V(X)$ is the set of vertices, $E(X)$ is the set of unoriented edges (digons) and $\vec{E}(X)$ is the set of oriented edges (arcs). Here, one can consider mixed graphs as digraphs, where both ways oriented edges considered as digons. Throughout this paper, an arc from the vertex $u$ to the vertex $v$ in $\vec{E}$ will be denoted by $\vec{uv}$. A digon between the verteces $u$ and $v$ will be denoted by $uv$. A graph that is obtained from a mixed graph by considering all arcs as digons is called the underlying graph of $X$ and will be denoted by $\Gamma(X)$. The degree of a vertex $u$ in $X$, denoted by $deg(u)$, is defined to be the degree of $u$ in $\Gamma(X)$. By a path (resp. a cycle, a walk) in a mixed graph, we mean a path (resp. a cycle, a walk) in $\Gamma(X)$. For an unoriented graph $(V(G),E(G))$, $L_G$ denotes the line graph of $G$. \\

For an unoriented graph $G$, the adjacency matrix of $G$, denoted by $A(G)=(a_{ij})$, is the square symmetric $(0,1)$-matrix whose rows and columns correspond to the vertices of $G$ and $a_{ij}=1$ if the vertices  $u_i$ and $u_j$ are adjacent. The Laplacian (resp. signless Laplacian) adjacency matrix of a graph $G$ is the square matrix $L(G)=D(G)-A(G)$ (resp. $Q(G)=D(G)+A(G)$) where $D(G)=diag(deg(u_i):u_i\in V(G))$. The incidence matrix $B=(b_{ij})$ of an undirected graph $G$ is the $(0,1)$-matrix where its rows correspond to the vertices of $G$, columns correspond to the edges of $G$ and the entries $b_{ij}$ equal $1$ if and only if the vertex $u_i$ and the edge $e_j$ are incident. The adjacency matrix of a digraph $D$, denoted by $A(D)=(a_{ij})$, is the same as the adjacency matrix of a graph with $a_{ij}=1$ if there is an arc from the vertex $u_i$ to the vertex $u_j$. Which means, the adjacency matrix of a directed graph $D$ is symmetric if and only if all arcs of $D$ are both ways oriented arcs.\\
Studying eigenvalues and eigenvectors of an adjacency matrix of a graph is one of the most important topics in algebraic graph theory; especially considering its many applications in combinatorics, chemistry and theoretical computer science \cite{Liu} and \cite{God}. Mixed graphs can be considered as a generalization of graphs. Further, they have more applications than graphs. Even though research about spectrum of mixed graphs is rare. One of the reasons for such rareness is that its adjacency matrix is not symmetric. Which means some of its eigenvalues are complex. Mohar defined an interesting general Hermitian adjacency matrix of mixed graphs as follows~\cite{Mohar2019ANK} :
Let $X$ be a mixed graph with $n$ vertices and $\alpha$ be the unit complex number $e^{i\theta}$  . Then, the $\alpha$-Hermitian adjacency matrix of $X$ is an $n\times n$ matrix $H_{\alpha}(X)=(h_{ij})$, where
 
\[h_{ij} = \left\{
\begin{array}{ll}
1 & \text{if }  u_iv_j \in E(D),\\
\alpha &  \text{if }  \vec{u_iv_j} \in \vec{E}(D), \\
\overline{\alpha} &  \text{if }  \vec{v_ju_i} \in \vec{E}(D),\\
0 & \text{otherwise}.
\end{array}
\right.
\]
Mohar focused in his study on the $\alpha$-Hermitian adjacency matrix where $\alpha$ is the primitive sixth root of unity~\cite{Mohar2019ANK}. While in an earlier study Guo and Mohar considered $\alpha$ to be the complex number $i$ and demonstrated many interesting spectral properties of the adjacency matrix $H_i$ \cite{BMI}. Abudayah et al. considered $\alpha$ to be the primitive third root of unity $\gamma$ instead of the primitive sixth root of unity~\cite{Abudayah2}. In fact, since $\gamma^2=\overline{\gamma}$ they discovered many interesting spectral properties of $H_\gamma$. \\
On the other hand, the traditional adjacency matrix and Laplacian adjacency matrix of a graph was extensively studied in literature, see for example \cite{zhang2011laplacian}. However, research about signless Laplacian adjacency matrix was sparse. Even though it was proven that studying graphs by its signless Laplacian adjacency matrix spectra is more efficient than studying them by their (adjacency) spectra \cite{CVETKOVIC2007155}. Considering the definition of signless Laplacian adjacency matrix of graphs, in this work, we define and study the $\gamma$-signless Laplacian adjacency matrix of mixed graphs. Accordingly, we need the following definitions and theorems which can be found in \cite{Abudayah2}. 
\begin{definition}
Let $X$ be a mixed graph and $H_\alpha$ be its $\alpha$-Hermitian adjacency matrix. Then, 
\begin{itemize}
\item The spectrum of the $\alpha$-Hermitian adjacency matrix of the mixed graph $X$, denoted by $\sigma_\alpha(X)$, is called $\alpha$-spectrum of $X$.
\item For any walk $W=v_1,v_2,\dots,v_n$, the $\alpha$-weight of $W$ is defined by:\[h_\alpha(W)=h_{v_1v_2}h_{v_2v_3}\dots h_{v_{n-1}v_{n}}.\]
\item $X$ is called an $\alpha$-monostore graph if the $\alpha$-weight of each cycle in $X$ equals one.
\end{itemize}
\end{definition}
\begin{theorem}
Let $X$ be a connected mixed graph. If $X$ is $\alpha$-monostore graph, then 
\[\sigma_\alpha(X)=\sigma_\alpha(\Gamma(X)).\]
\end{theorem}
\section{The $\gamma$-Incidence Matrix of Mixed Graphs}
In this section, our aim is to define an incidence matrix for mixed graphs. 
Let $X$ be a mixed graph, with $n=|V(X)|$, $m=|E(\Gamma(X))|$ and $\gamma$ be the primitive third root of unity $e^{\frac{2\pi}{3}i}$. Define the $\gamma$-incidence matrix of the mixed graph $X$, to be the $n\times m$ matrix $B_\gamma=[b_{ij}]$, where 
\[b_{ij} = \left\{
\begin{array}{ll}
1 & \text{if the vertex } v_i \text{ incident to the digon }   e_j \in E(X),\\
\gamma &  \text{if the vertex } v_i \text{ is the terminal vertex of the arc }   e_j \in \vec{E}(X), \\
\gamma^2 &  \text{if the vertex } v_i \text{ is the initial vertex of the arc }   e_j \in \vec{E}(X),\\
0 & \text{otherwise}.
\end{array}
\right.
\]
The above definition is consistent with the graphs incidence matrix definition. Another aspect of consistency is that since the structure of $\gamma$-incidence matrix of a mixed graph $X$ is similar to the structure of the incidence matrix of its underlying graph $\Gamma(X)$, and $\{1,\gamma,\gamma^2\}$ is closed under multiplication, once can easily observe the following:
\begin{observation}\label{ob1}
Let $X$ be a mixed graph, $H_\gamma(X)$ be its $\gamma$-Hermitian adjacency matrix and $B_\gamma(X)$ be its $\gamma$-incidence matrix. Then,
\begin{itemize}
\item $B_\gamma(\Gamma(X))=B(\Gamma(X))$.
\item $B_\gamma B_\gamma^*=D+H_\gamma(X)$.
\item $B_\gamma^*B_\gamma =2I+H_\gamma(\mathbb{X})$.
\end{itemize}
where $D=D(X)$, $I$ is the identity matrix and $H_\gamma(\mathbb{X})$ is the $\gamma$-adjacency matrix of a mixed graph $\mathbb{X}$.
\end{observation}
It should be mentioned here that the underlying graph of the mixed graph $\mathbb{X}$ is the line graph of the underlying graph of $X$. The following definition paves the way to clarify the structure of the mixed graph $\mathbb{X}$, see Figure \ref{fig1}.

\begin{definition}
For a mixed graph $X$, we define the algebraic line mixed graph of the mixed graph $X$, denoted by $AL_X$, to be the mixed graph whose vertex set $E(X)\cup \vec{E}(X)$ and a set of arcs and digons as follows:
\begin{itemize}
\item An arc from the arc $e_i$ to the arc $e_j$, if the terminal vertex of $e_i$ is the initial vertex of the arc
$e_j$.
\item An arc from the digon $e_i$ to the arc $e_j$ if the terminal vertex $e_j$ is an end vertex of the digon $e_i$.
\item An arc from the arc $e_j$ to the digon $e_i$ if the initial vertex of $e_j$ is an end vertex of the digon $e_i$.
\item A digon between the arcs $e_i$ and $e_j$ if $e_i$ and $e_j$ have the same initial vertex or end vertex.
\item A digon between the digons $e_i$ and $e_j$ if $e_i$ and $e_j$ have common vertex.
\end{itemize}
\end{definition} 

Using the above definition together with Observation \ref{ob1}, one can easily check that $AL_X=\mathbb{X}$. Thus,
\begin{equation}\label{5}
B_\gamma^*B_\gamma=2I + H_\gamma(AL_X).
\end{equation}
Furthermore, it is obvious that the algebraic line mixed graph $AL_X$ of a mixed graph $X$ satisfies the following:
\begin{equation}\label{6}
\Gamma(AL_X)=L_{\Gamma(X)}.
\end{equation}

\begin{figure}[H] 
\begin{subfigure}{1\textwidth}
  \centering
  \includegraphics[width=.5\linewidth]{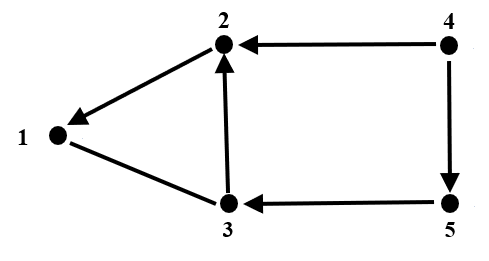}  
  \caption{The mixed graph $X$.}
  \label{fig:sub-first}
\end{subfigure}
\begin{subfigure}{1\textwidth}
 \centering
  \includegraphics[width=.5\linewidth]{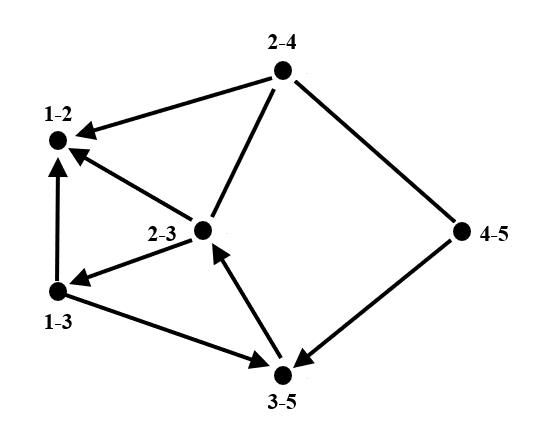}  
  \caption{The Algebraic Line mixed graph $AL_X$.}
  \label{fig:sub-second}
\end{subfigure}
\caption{ }
\label{fig1}
\end{figure}
Similar to the definition of signless Laplacian adjacency matrix of unoriented graphs, we define $\gamma$-signless Laplacian adjacency matrix of mixed graphs as follows: Let $X$ be a mixed graph, $D=diag\{v_1,v_2,\dots,v_n\}$ and $H_\gamma$ be its Hermitian adjacency matrix then we define the signless Laplacian adjacency matrix of $X$ by: \[Q_\gamma=B_\gamma B_\gamma^* = D+H_\gamma.\]

\section{The $\gamma$-Signless Laplacian Adjacency Matrix}

Let $X$ be a mixed graph and $Q_\gamma$ be its  signless Laplacian adjacency matrix. Then, the characteristic polynomial of the $\gamma$-signless Laplacian (resp. Hermitian) adjacency matrix  of $X$ will be denoted by $\chi_Q(X,\lambda)$ (resp. $\chi_H(X,\mu)$). The eigenvalues of the matrix $Q_\gamma$ (resp. $H_\gamma$) are called the $Q$-eigenvalues (resp. $H$-eigenvalues) of $X$. The $Q$-eigenvalues and $H$-eigenvalues (which are all real) of the mixed graph $X$ will be denoted by $\lambda_1\ge \lambda_2\ge \dots \ge \lambda_n$ and  $\mu_1\ge \mu_2\ge \dots \ge \mu_n$ respectively. \\
Now, since $Q_\gamma=B_\gamma B_\gamma^*$ is positive semidefinite, all $Q$-eigenvalues of $X$ should be non-negative. Also, since the non-zero eigenvalues of $B_\gamma B_\gamma^*$ and $B_\gamma^* B_\gamma$ are the same, using Equations \ref{5} and \ref{6}, one can immediately obtain the following:
\begin{equation}
\chi_H(AL_X,\lambda)=(\lambda+2)^{m-n}\chi_Q(X,\lambda+2),
\end{equation}
where, $m=|E(X)\cup \vec{E}(X)|$ and $n=|V(X)|$. \\
Since the diagonal entries of the $\gamma-$signless Laplacian adjacency matrix of a mixed graph $X$ are the degrees of the vertices of $X$, we get $tr(Q_\gamma)=2m$, where $m$ is the number of arcs and digons in $X$. Therefore, \[ \sum_{i=1}^n\lambda_i=2m.\]
\begin{theorem}\label{lambda2}
Let $X$ be a mixed graph with $|V(X)|=n$, $|\vec{E}(X)\cup E(X)|=m$ and $Q_\gamma$ be its signless Laplacian adjacency matrix. Then,
\[\sum_{i=1}^n\lambda_i^2=2m+\sum_{uv\in \vec{E}(X)\cup E(X)}(deg(u)+deg(v)).
\]
\end{theorem}
\begin{proof}
Observing that, $\sum_{i=1}^n\lambda^2=tr(Q_\gamma^2)$ we get,
\begin{align}
\sum_{i=1}^n{\lambda_i^2}&=tr\left((B_\gamma^*B_\gamma)^2\right)\\
&=tr\left((2I+H_\gamma(AL_X))^2\right)\\
&=4m+\sum_{e\in V(AL_X)}deg_{AL_X}(e)\\
&=4m+\sum_{uv\in \vec{E}(X)\cup E(X)}\left(deg(u)+deg(v)-2\right)\\
&=2m+\sum_{uv\in \vec{E}(X)\cup E(X)}\left(deg(u)+deg(v)\right),
\end{align}
where $deg_{AL_X}(e)$ is the degree of the vertex $e$ in $AL_X$.
\end{proof}\\
For a connected graph $G$, it is well known that the signless Laplacian adjacency matrix is singular if and only if $G$ is a bipartite graph \cite{CVETKOVIC2007155}. The following is a generalization of this theorem.
\begin{theorem}
Let $X$ be a mixed graph and $Q_\gamma$ be its signless Laplacian adjacency matrix. Then, $Q_\gamma$ is singular if and only if $X$ is a $\gamma$-monostore graph and $\Gamma(X)$ is bipartite.
\end{theorem}
\begin{proof}
Suppose that $x=[x_u]_{u\in V(X)}$ be a non-zero vector, with $Q_\gamma x=0$. Then, $B_\gamma^* x=0$ and thus, for every arc $\vec{uv}$ or digon $uv$ in $X$ we have, $x_v=-x_uh_{uv}$. Now, let $C=u_1u_2 \dots u_{k-1}u_1$ be a cycle in $X$. Then, 
\begin{align}\label{mono}
x_{u_1} &=(-1)^kh_{u_1u_2}h_{u_2u_3}\dots h_{u_{k-1}u_1}x_{u_1}\\
            &=(-1)^kh_\gamma(C)x_{u_1}.
  \end{align}
Observing that $X$ is weakly connected graph, $x_{u_1}\ne 0$. Therefore,\[ 1=(-1)^kh_\gamma(C).\]
The fact that this equation is true  if and only if $h_\gamma(C)=1$ and $k$ is even concludes the proof.
\end{proof}\\
For a mixed graph $X$ it has been proven that spectral radius of the $\gamma$-Hermitian adjacency matrix $H_\gamma$ of $X$ is less than the maximum degree of the vertices of $X$. To be more formal,
\begin{equation}\label{spec}
\max_{i}|\mu_i| \le \max_{u\in V(X)}\{deg(u)\}.
\end{equation}
 An immediate consequence of this result is that, the largest $Q$-eigenvalue of $X$ should be less than twice the maximum degree of the vertices of $X$. The following theorem gives a sharper upper bound of the largest $Q$-eigenvalues of $X$.
\begin{theorem}
Let $X$ be a mixed graph and $Q_\gamma$ be its signless Laplacian adjacency matrix. Then,
\[ \displaystyle \max_{u\in V(X)}\left\lbrace deg(u)  \right\rbrace  \le \lambda_1(X) \le  \displaystyle \max_{uv\in \vec{E}(X)\cup E(X)}\left\lbrace deg(u)+deg(v) \right\rbrace.
\]
\end{theorem}
\begin{proof}
Observing that the non-zero eigenvalues of $B_\gamma B_\gamma^*$ and $B_\gamma^* B_\gamma$ are the same, $\lambda_1(X)=\mu_1(AL_X)+2$. Using inequality \ref{spec} we have 
\begin{align}
\mu_1 &\le \displaystyle \max_{e\in V(AL_X)}\left\lbrace deg(e) \right\rbrace\\
            &=\displaystyle \max_{uv\in \vec{E}(X)\cup E(X)}\left\lbrace deg(u)+deg(v)-2 \right\rbrace.
  \end{align}
Therefore, $\lambda_1 \le \displaystyle \max_{uv\in \vec{E}(X)\cup E(X)}\left\lbrace deg(u)+deg(v) \right\rbrace$. \\
For the other side, observing that for every vector $x\in \mathbb{C}^n$ with $||x||=1$ we get $x^*Q_\gamma x \le \lambda_1$. Therefore, for every vertex $u$ of $X$ we have $deg(u) \le \lambda_1$.
\end{proof}\\

Now, let $M$ be a Hermitian matrix. Then, the spectral theorem says that $M$ has a set of orthonormal basis $\{x_i\}_{i=1}^{n}$ of its eigenvectors. Which means that there is a unitary matrix $U$ such that \[U^*MU = diag\{\lambda_1,\lambda_2,\dots,\lambda_n\}.\]
Therefore,  
\begin{equation}\label{diag}
M=\sum_{i=1}^n \lambda_i P_i,
\end{equation}
where $P_i=x_ix_i^*$. It is obvious here that the matrix $P_i$ is a Hermitian idempotent matrix, that is $P_i=P_i^*=P_i^2$. Moreover, since $UU^*=I$, we have $\sum_{i=1}^nP_i=I$. Which means,
\begin{align}
  \sum_{i=1}^n(P_i)_{rs}=0,  \hspace{25 pt} \text{         if $r\ne s$} \label{1}\\
  \sum_{i=1}^n(P_i)_{rs}=1, \hspace{25 pt} \text{         if $r = s$} \label{2}
\end{align}
Also, since $\sum_{i, j}\overline{x_i}x_j=\sum_{i=1}^n\overline{x_i}\sum_{j=1}^n x_j$ and $\sum_{i=1, i\ne j}^n|x_i|^2|x_j|^2=|x_j|^2(1-|x_j|^2)$, we can easily observe the following theorem:
\begin{theorem}\label{vec}
For any vector  $x=[x_1,x_2,\dots,x_n]^T$, if $x^*x=1$, the following holds:
\begin{align}
 \sum_{i, j}\overline{x_i}x_j\ge 0\label{3}\\
  \sum_{i=1, i\ne j}^n|x_i|^2|x_j|^2\le \frac{1}{4} \label{4}
  \end{align}
\end{theorem}
Now, we are ready to give some upper and lower bound of the extreme eigenvalues of the $\gamma$-signless Laplacian adjacency matrix of mixed graphs.
\begin{theorem}\label{ineq1}
Let $X$ be a mixed graph, $n=|V(X)|$, $m=|E(\Gamma(X))|$ and $Q_\gamma$ its signless Laplacian adjacency matrix. Then, 
\begin{equation}
\frac{n}{2}\lambda_n+1 \le m \le \frac{n}{2}\lambda_1-1.
\end{equation}
\end{theorem}
\begin{proof}
Using Equation \ref{diag} and Equation \ref{1} we have:
\[ Q_\gamma=\sum_{i=1}^n\lambda_iP_i \hspace{15 pt} \text{     and for     } r\ne s, \hspace{15 pt} \sum_{i=1}^n \lambda_i(P_i)_{rs}=0.
\]
Without loss of generality, we may assume that $X$ has a digon $uv$. Then,
$ 1=(Q_\gamma)_{uv}=\overline{(Q_\gamma)_{uv}}.$ Therefore,
\[\sum_{i=1}^n\lambda_iRe((P_i)_{uv})=1,
\]
Setting 
\begin{equation}\label{ci}
c_i=\frac{1\pm 2Re((P_i)_{uv})}{n},
\end{equation}
we get
 \begin{align}\label{tr}
\sum_{i=1}^nc_i\lambda_i &=\frac{tr(Q_\gamma)}{n}\pm \frac{2\sum_{i=1}^nRe((P_i)_{uv})\lambda_i}{n}\\
            &=\frac{2m}{n}\pm \frac{2}{n}.
  \end{align} 
On the other hand, using Equation \ref{2}, for $r\in V(X)$ we have
\begin{align}
1 &=\sum_{i=1}^n(P_i)_{rr}=\\
            &=\sum_{i=1}^n(x_i)_r\overline{(x_i)_r}\\
          &=\sum_{i=1}^n|(x_i)_r|^2.
  \end{align} 
Therefore, using Theorem \ref{vec} we have,
 \[ \sum_{r,s\in V(X),r\ne s}|(x_i)_r|^2|(x_i)_s|^2\le \frac{1}{4}.
\]
And thus, $|(P_i)_{rs}| \le \frac{1}{2}$. Which means,
\[ -\frac{1}{2} \le Re((P_i)_{rs}) \le \frac{1}{2}.
\]
So,
\[ 0\le c_i \le \frac{2}{n}
\]
Now, it can be easily seen that $\sum_{i=1}^nc_i=1$. Therefore, 
\[ \lambda_n \le \sum_{i=1}^nc_i\lambda_i \le \lambda_1.
\]
Finally, the fact that $tr(Q_\gamma)=2m$ ends the proof.\\
\end{proof}\\
An immediate consequence of the above theorem is the following corollary:
\begin{corollary}
Let $X$ be a mixed graph, $n=|V(X)|$ and $Q_\gamma$ its signless Laplacian adjacency matrix. Then,
\begin{equation}\label{cor1}
\lambda_1-\lambda_n\ge \frac{4}{n}.
\end{equation}
\end{corollary}
Obviously, the lower bound of spread of $Q_\gamma$ in Inequality \ref{cor1} is small. However, the following theorem gives a refinement of this lower bound.

\begin{theorem}
Let $X$ be a mixed graph, $n=|V(X)|$ and $Q_\gamma$ its signless Laplacian adjacency matrix. Then,
\begin{equation}
\lambda_1-\lambda_n\ge 2.
\end{equation}
\end{theorem} 
\begin{proof}
Let $uv$ be any arc or digon in $X$. Then
\begin{align}
1 &= |\left(Q_\gamma \right)_{uv}|\\
&=\left| \sum_{i=1}^n \lambda_i (P_i)_{uv} \right|.
\end{align}
Observing that $\sum_{i=1}^n (P_i)_{uv}=0$, we have 
\begin{align}
1&=\left| \sum_{i=1}^n (\lambda_i-\frac{\lambda_1+\lambda_n}{2} (P_i)_{uv} \right|\\
&\le  \sum_{i=1}^n\left| (\lambda_i-\frac{\lambda_1+\lambda_n}{2}\right| \left|(P_i)_{uv} \right|\\
&\le \frac{\lambda_1-\lambda_n}{2}\sum_{i=1}^n\left|(P_i)_{uv} \right|.
\end{align}
Now, since $P_i$ is positive semidefinite we have,
\[ \left|(P_i)_{uv} \right| \le \sqrt{(P_i)_{uu}(P_i)_{vv}} \le \frac{(P_i)_{uu}+(P_i)_{vv}}{2}.
\]
Therefore,
\begin{align}
\sum_{i=1}^n\left|(P_i)_{uv} \right|&\le \sum_{i=1}^n \frac{(P_i)_{uu}+(P_i)_{vv}}{2}\\
&=1.
\end{align}
Which ends the proof.
\end{proof}\\
The following is a sharper inequality for the right hand side of Theorem \ref{ineq1}.
\begin{theorem}\label{digon}
Let $X$ be a mixed graph, $k=|\vec{E}(X)|$, $l=|E(X)|$ and $Q_\gamma$ be its signless Laplacian adjacency matrix. Then,
\[ \lambda_1 \ge \frac{4l+k}{n}.
\]
\end{theorem}
\begin{proof}
First, note that
\begin{align}
\sum_{u,v\in V(X)}(Q_\gamma)_{uv}&=\sum_{u\in V(X)}deg(u)+\sum_{u,v\in V(X),u\ne v}(Q_\gamma)_{uv}\\
&=2(k+l)-k+2l\\
&=4l+k
\end{align}
On the other hand,
\begin{align}
\sum_{u,v\in V(X)}(Q_\gamma)_{uv}&=\sum_{u,v\in V(X)}\sum_{i=1}^n \lambda_i (P_i)_{uv}\\
&=\sum_{u\in V(X)}\sum_{i=1}^n \lambda_i (P_i)_{uu}+\sum_{u,v\in V(X),u\ne v}\sum_{i=1}^n \lambda_i (P_i)_{uv}\\
&=\sum_{i=1}^n\lambda_i\sum_{u\in V(X)}(P_i)_{uu}+\sum_{i=1}^n\lambda_i\sum_{u,v\in V(X),u\ne v}(P_i)_{uv}.
\end{align}
But, using Equation \ref{1} we have $\sum_{u,v\in V(X),u\ne v}(P_i)_{uv}=0$. Therefore,
\begin{align}
\sum_{u,v\in V(X)}(Q_\gamma)_{uv}&=\sum_{i=1}^n\lambda_i\sum_{u\in V(X)}(P_i)_{uu}\\
&\le \lambda_1 \sum_{i=1}^n\sum_{u\in V(X)}(P_i)_{uu}\\
&=n\lambda_1
\end{align}
Therefore,
$\lambda_1 \ge \frac{4l+k}{n}.$
\end{proof}\\
One can easily check that the above inequality is sharp when $X$ is regular un-oriented graph.
\begin{example}
\begin{figure}[H] 
\begin{subfigure}{1\textwidth}
  \centering
  \includegraphics[width=.5\linewidth]{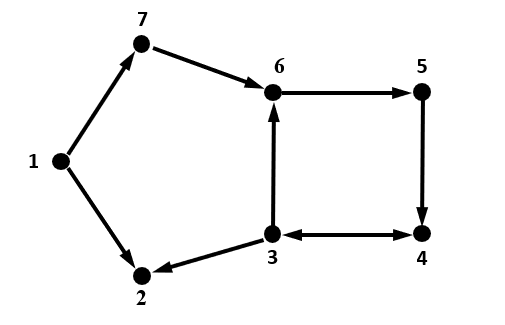}  
  \caption{The mixed graph $X$.}
  \label{fig:sub-firstb}
\end{subfigure}
\begin{subfigure}{1\textwidth}
 \centering
  \includegraphics[width=.5\linewidth]{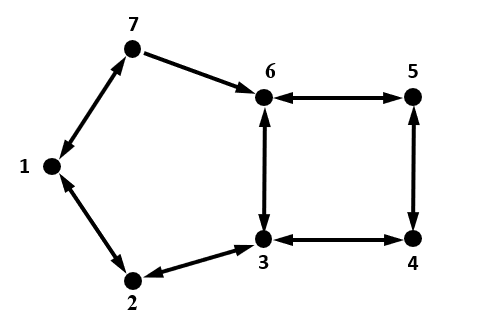}  
  \caption{The mixed graph $Y$}
  \label{fig:sub-secondb}
\end{subfigure}
\caption{The mixed graphs $X$ and $Y$ are cospectral }
\label{fig2}
\end{figure}
Consider the mixed graphs $X$ and $Y$ shown in Figure \ref{fig2}. Applying Theorem \ref{digon} one can get $\lambda_1(X)\ge 1.57$. Note that a better lower bound can be obtained using switching equvilance idea. To clarify that let $X'=X-{\vec{76}}$ and $Y'=Y-{\vec{76}}$. Then, obviously $Q_\gamma(X')$ and $Q_\gamma(Y')$ are similar with similarity matrix
\[S=\left(
\begin{array}{ccccccc}
 1 & 0 & 0 & 0 & 0 & 0 & 0 \\
 0 & \gamma & 0 & 0 & 0 & 0 & 0 \\
 0 & 0 & 1 & 0 & 0 & 0 & 0 \\
 0 & 0 & 0 & 1 & 0 & 0 & 0 \\
 0 & 0 & 0 & 0 & \gamma^2 & 0 & 0 \\
 0 & 0 & 0 & 0 & 0 & \gamma & 0 \\
 0 & 0 & 0 & 0 & 0 & 0 & \gamma \\
\end{array}
\right)\]
Therefore, the mixed graphs $X$ and $Y$ are $Q_\gamma$ cospectral. Thus, a better lower bound, $\lambda_1\ge 4.142$ can be obtained when appling Theorem  \ref{digon} for the mixed graph $Y$.
\end{example}

Finally, we want to point out that the technique used in proving Theorem \ref{ineq1} can be utilize  to obtain other upper bounds as well as lower bounds for the $\lambda_1$ and $\lambda_n$. For example consider the following theorem, which can be found in \cite{dragomir2003survey} page 72.

\begin{theorem}[The Cassels’ Inequality]
If the positive real sequences $a = (a_1,a_2, . . . , a_n)$ and $b = (b_1, b_2, . . . , b_n)$ satisfy the condition:
\begin{equation}
0<m\le \frac{a_k}{b_k}\le M <\infty, \text{ for all  } k.
\end{equation}
and $c = (c_1,c_2, . . . , c_n)$ is a sequence of non-negative real numbers. Then, 
\begin{equation}
\frac{\left( \sum_{i=1}^nc_ia_i^2\right)\left( \sum_{i=1}^nc_ib_i^2\right)}{\left( \sum_{i=1}^nc_ia_ib_i\right)^2}\le \frac{(M+m)^2}{4Mm}
\end{equation}
\end{theorem}
Now, Setting $a_i=\lambda_i$, $c_i=\frac{1\pm 2Re((P_i)_{uv})}{n}$ and $b_i=1$ for $i=1,2,\dots,n$ we get,
\begin{equation}\label{cass}
\frac{\left( \sum_{i=1}^nc_i\lambda_i^2\right)}{\left( \sum_{i=1}^nc_i\lambda_i\right)^2}\le \frac{(\lambda_1+\lambda_n)^2}{4\lambda_1\lambda_n}.
\end{equation}
Therefore, 
\begin{equation}
n\frac{tr(Q_\gamma^2)+2}{(2m-2)^2}\le \frac{(\lambda_1+\lambda_n)^2}{4\lambda_1\lambda_n}.
\end{equation}
Finally, observe that $Q_\gamma^2=\sum_i \lambda_i^2 P_i$, we get the following theorem:
\begin{theorem}
Let $X$ be a mixed graph and $Q_\gamma$ be its Laplacian adjacency matrix. If $Q_\gamma$ is non-singular, then
\[
\frac{2m+2+\sum_{uv\in \vec{E}(X)\cup E(X)}(deg(u)+deg(v))}{(m-1)^2} \le \frac{(\lambda_1+\lambda_n)^2}{n\lambda_1 \lambda_n}
\]
\end{theorem}

\bibliography{signless10}

\end{document}